\documentclass[12pt,reqno]{amsart}

\usepackage[T1]{fontenc}
\usepackage[utf8]{inputenc}
\usepackage{amsmath,amsfonts,amssymb,amsthm,mathtools}
\usepackage{tikz}
\usepackage{hyperref}
\hypersetup{hidelinks}
\usepackage{microtype}

\numberwithin{equation}{section}

\theoremstyle{plain}
\newtheorem{theorem}{Theorem}[section]
\newtheorem{lemma}[theorem]{Lemma}
\newtheorem{proposition}[theorem]{Proposition}

\theoremstyle{definition}
\newtheorem{definition}[theorem]{Definition}
\newtheorem{question}[theorem]{Question}

\newtheorem{remark}[theorem]{Remark}

\title[Why Does the Slope Definition of $e$ Work?]
{Why Does the Slope Definition of $e$ Work?:\\
A Self-Contained Exploration After Real Analysis}
\author{H\`ung Vi\d{\^e}t Chu}
\address{Department of Mathematics, Washington and Lee University, Lexington, VA 24450, USA}
\email{hchu@wlu.edu}

\author{Steven J. Miller}
\address{Department of Mathematics and Statistics, Williams College, Williamstown, MA 01267, USA}
\email{sjm1@williams.edu}

\author{Joshua M. Siktar}
\address{Department of Mathematics, Texas A\&M University, College Station, TX 77843, USA}
\email{jmsiktar@tamu.edu}


\begin{document}

\begin{abstract}
Euler's number $e$ is often introduced through compound interest, an infinite series, an area, or the slope of an exponential function.  The last viewpoint is especially attractive: choose the base $b\in (2,3)$ so that the graph of $y=b^x$ has tangent slope $1$ at $(0,1)$.  However, a curious student may raise several questions that a brief introduction necessarily leaves implicit. What does $b^x$ mean when $x$ is irrational? Why is $b^x$ differentiable at $x = 0$?  How can one  show that the slopes corresponding to bases $2$ and $3$ lie on opposite sides of slope $1$? Finally, why is there a unique base whose slope is exactly $1$? The paper provides an answer using tools from a first course in real analysis, including limits, continuity, derivatives, and the Intermediate Value Theorem. Its primary audience is undergraduate students familiar with standard concepts from an introductory analysis course who wish to apply what they have learned to explore questions left open in earlier courses. A secondary audience is instructors who may use selected portions as supplementary material in an introduction-to-proofs course. The mathematical ingredients are classical; our aim is to organize them into a self-contained exposition and to make explicit the logical steps that are often compressed in standard treatments.
\end{abstract}

\maketitle
\tableofcontents

\section{Why another introduction to $e$?}\label{sec:introduction}

Euler's number
$$e \ \approx\ 2.718281828459045\ldots$$
is one of the most recognizable constants in mathematics.  Together with $\pi$ and the golden ratio $\varphi$, it is among the constants that students often recognize long before they study its theory.  The constants $e$ and $\pi$ meet in Euler's identity
\begin{equation}
e^{i\pi}+1\ =\ 0,
\end{equation}
a formula celebrated for bringing together the constants $0$, $1$, $e$, $\pi$, and $i = \sqrt{-1}$ in a single line.

Students encounter $e$ in several apparently different ways.  One comes from finance.  If a principal amount $P$ is invested for one year at annual interest rate $r$, and the interest is compounded $n$ times during the year, then the balance is
$$
 P\left(1+\frac{r}{n}\right)^n.
$$
Taking $P=r=1$ and allowing the compounding frequency to increase leads to
\begin{equation}\label{eq:compound} e\ =\ \lim_{n\to\infty}\left(1+\frac1n\right)^n.\end{equation}
This is why the product limit is naturally connected with ``continuous compounding,'' rather than being an isolated limit invented for calculus.

From power series, one meets
\begin{equation}\label{eq:series}
 e\ =\ \sum_{n=0}^{\infty}\frac1{n!}.
\end{equation}
The binomial theorem already hints at the connection between \eqref{eq:compound} and \eqref{eq:series}:
$$
\left(1+\frac1n\right)^n\ =\ \sum_{k=0}^n \frac{n(n-1)\cdots(n-k+1)}{k! n^k},
$$
and, for each fixed $k$, the  coefficient
$$\frac{n(n-1)\cdots(n-k+1)}{n^k}$$
tends to $1$ as $n\to\infty$. 

The series description immediately gives the familiar estimate $2<e<3$.  Indeed,
$$e\ >\ 1+1\ =\ 2,$$
and, since 
$$(3+j)!\ = \ 3! (3+1)(3+2)\cdots (3+j)\ \ge\  3! 3^j\mbox{ for }j\ge 1,$$
we have
$$e\ =\ 1+1+\frac1{2!}+\frac{1}{3!} + \sum_{j=1}^{\infty}\frac1{(3+j)!}
 \ \le\ \frac83+\frac1{3!}\sum_{j=1}^{\infty}\frac1{3^j}
 \ =\ \frac{11}{4}\ <\ 3.$$
This estimate is not part of our later construction of $e$, but it helps explain why bases $2$ and $3$ are natural integers to compare.

A third characterization, once integration is available, is
\begin{equation}\label{eq:area}
 \int_1^e \frac{1}{x}\,dx\ =\ 1.
\end{equation}
These viewpoints are equivalent once enough tools have been developed, but their equivalence is not immediate from the definitions themselves.

Stewart \cite{Ste} presents another viewpoint that is especially well-suited to a course centered on continuity and derivatives.  Consider the family of exponential functions
$$y\ =\ b^x,\quad b>0.$$
Every graph passes through $(0,1)$.  Numerical experimentation suggests that the tangent slope there is about $0.69$ when $b=2$ and about $1.10$ when $b=3$.  For example, a calculator gives
$$\frac{2^{0.001}-1}{0.001}\ \approx\ 0.6934\quad\mbox{ and }\quad\frac{3^{0.001}-1}{0.001}\ \approx\ 1.0992.$$
These computations are only a preview: they do not prove that either derivative exists.  They do, however, suggest that there should be a base between $2$ and $3$ for which the slope is exactly $1$. See Figure \ref{fig:exp-family}.

Before writing down the derivative, there is a point worth making explicit: what does $b^h$ mean when $h$ is irrational?  Integer and rational powers can be defined algebraically.  A standard construction extends them continuously to real exponents.  For example, one may choose rational numbers $r_n\to x$ and define
$$b^x\ :=\ \lim_{n\to\infty}b^{r_n}.$$
One must then prove that the limit is independent of the chosen rational approximations and that the familiar exponent laws continue to hold. We summarize this construction in Appendix~\ref{app:real-exponents}; see also Badi \cite{Badi} and Eidswick \cite{Eidswick}.  The continuous extension gives
\begin{equation}\label{eq:exp-laws}
 b^0\ =\ 1,\quad b^{x+y}\ =\ b^x b^y, \mbox{ for all real numbers }x, y.
\end{equation}

We recall the slope formula 
\begin{equation}\label{eq:defg}
 g(b)\ :=\ \lim_{h\to0}\frac{b^h-1}{h},
\end{equation}
provided the limit exists.  If it does, then $g(b)$ is precisely the derivative of $b^x$ at $x=0$.  For $b>1$, the right-hand quotients
$$
\frac{b^h-1}{h},\quad h>0,
$$
increase when the base $b$ increases, so one already expects the slope $g(b)$, if it exists, to increase with the base.  We will later prove the stronger statement that $g$ is continuous and strictly increasing on $(0,\infty)$. 

\begin{figure}[ht]
\centering
\begin{tikzpicture}[x=3.0cm,y=1.2cm]
  \draw[->] (-0.35,0) -- (1.20,0) node[right] {$x$};
  \draw[->] (0,-0.5) -- (0,3.65) node[above] {$y$};
  \draw[thick,dashed,domain=-0.25:1.05,samples=80,smooth] plot (\x,{exp(ln(2)*\x)});
  \draw[thick,domain=-0.25:1.05,samples=80,smooth] plot (\x,{exp(1*\x)});
  \draw[thick,dotted,domain=-0.25:1.05,samples=80,smooth] plot (\x,{exp(ln(3)*\x)});
  \fill (0,1) circle (1.2pt);
  \node[left] at (0,1) {$(0,1)$};
  \node[right] at (1.01,2.05) {$2^x$};
  \node[right] at (0.96,2.72) {$b^x$ with $b\approx2.72$};
  \node[right] at (0.88,3.18) {$3^x$};
\end{tikzpicture}
\caption{A numerical preview.  The middle base appears to have tangent slope close to $1$ at $(0,1)$.  The picture motivates the argument; it does not replace it.}
\label{fig:exp-family}
\end{figure}
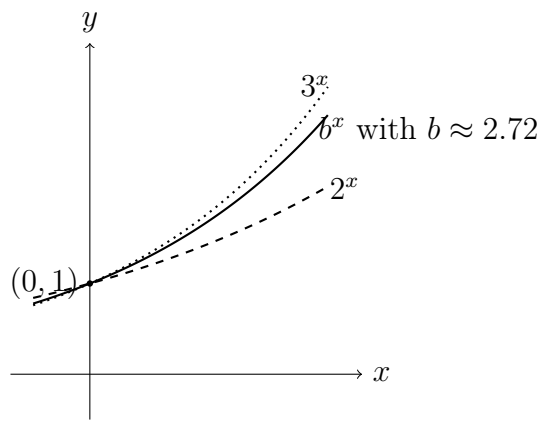

The introduction of $e$ using slopes immediately raises three natural questions.

\begin{question}\label{q:existence}
For each $b>0$, why does the limit in \eqref{eq:defg} exist?
\end{question}

\begin{question}\label{q:bounds}
How can we prove
$$g(2)\ <\ 1\ <\ g(3)?$$
\end{question}

\begin{question}\label{q:ivt}
How do we know that some $b$ satisfies $g(b)=1$?  The Intermediate Value Theorem (IVT) suggests the answer, but first we must verify that $g$ is continuous.  Moreover, if such a base exists, is it unique? Why or why not? 
\end{question}

The issue of real exponents matters pedagogically.  Students may feel that they already know how to differentiate powers, but the problems
$$
 \frac{d}{dx}x^3,\quad \frac{d}{dx}x^{3/2},\quad \frac{d}{dx}x^{\sqrt3}
$$
are not identical.  The first can be handled directly with the binomial theorem, and the second can be reduced to an algebraic relation such as $y^2=x^3$, while the third forces us to confront what an irrational exponent means.  Appendix~\ref{app:three-powers} works through these three methods in more detail.

\subsection{Audience and prerequisites}\label{subsec:audience}

Our primary audience is undergraduate students who are familiar with standard concepts and results from a first course in analysis, including limits, continuity, derivatives, and the Intermediate Value Theorem.  The appendices supply optional background on real exponents and on the different methods used for integer, rational, and irrational powers.   The article is also intended for instructors seeking supplementary reading that uses questions about the number $e$ as the framework to introduce students to rigorous mathematical proofs.

One point we wish to emphasize throughout this article is the importance and joy of mathematical discovery and inquisitive thinking. An important role of instructors is to foster students' questioning, which can lead to deeper understanding and even research projects \cite{CMS}. In our experience, students often find mathematics most rewarding when ideas from a higher-level course allow them to revisit and deepen their understanding of material encountered earlier. Learning becomes especially meaningful when new concepts connect with, explain, or illuminate previous ones, gradually forming a coherent network of mathematical knowledge.

Finally, knowing an answer is different from understanding why it is true. A completed crossword puzzle may be easy to check, but simply seeing the filled-in answers does little to prepare you to solve the next one. Other common examples include a completed Sudoku or jigsaw puzzle. Mathematics is no different: seeing a finished argument is not the same as discovering why it works. This distinction reflects the well-studied contrast between active and passive learning.

\subsection{Earlier approaches and why revisit the topic?}\label{subsec:literature}

The problem of how to introduce $e$ is classical.  Menger \cite{Menger} explicitly approached it as a problem in the teaching of calculus: he compared the family $y=a^x$ with the line $y=x+1$ and used the geometry of the intersections to motivate the unique base whose tangent at $(0,1)$ has slope $1$.  Tull \cite{Tull} gave a brief discovery-oriented treatment in which students observe from graphs that the derivative at the origin appears to increase with the base and hence that there ``ought to be'' a distinguished base with slope $1$. 

Meanwhile, Samelson \cite{Samelson} supplied a concise rigorous route through convexity.  His key observation is that convexity makes the secant quotient $(a^h-1)/h$ monotone in $h$, giving the existence of the derivative at the origin.  Eidswick \cite{Eidswick} pursued the same differentiability question by a different route, proving monotonicity first on rational arguments and then passing to the reals by continuity.  Badi \cite{Badi} later developed this general program in substantially greater analytic detail, including a construction of real powers and a systematic development of exponential and logarithmic functions.  More recently, Meintrup and Singer \cite{MeintrupSinger} proposed a different elementary route based on a family of integrals approximating the logarithm.

Why, then, another paper on $e$? Our contribution is not a new definition of $e$, a new proof of differentiability, or a new use of convexity. Rather, we aim to make this classical route accessible from the perspective of undergraduate students who have just completed an introductory course in real analysis. We present rigorous mathematical proofs in a conversational and exploratory manner, rather than as a sequence of formulas and conclusions. The short notes of Tull, Samelson, and Eidswick are elegant, but they move at a pace appropriate for more advanced students. Badi’s treatment is more general, developing exponential and logarithmic functions in full generality rather than focusing specifically on $e$. While this broader perspective has clear advantages, its generality may also be somewhat intimidating to students encountering the subject for the first time. Our approach instead keeps the number $e$ at the center of the discussion.  Our focus is on the three Questions \ref{q:existence}, \ref{q:bounds}, and \ref{q:ivt} arising from the introduction of $e$. We answer them in the order a student is likely to encounter them, make the hidden prerequisites explicit, and separate motivational observations from proofs.  The resulting argument shows not only that the familiar definition works, but why each logical step is needed. We also stress that this exposition introduces mathematical techniques students will likely see elsewhere, including density arguments to extend results on $\mathbb{Q}$ to $\mathbb{R}$ and careful limit computations to prove existence of derivatives.

The central mathematical facts are therefore classical; the pedagogical point is the order and level of explanation.  Convexity answers the existence question, elementary secant estimates locate the desired slope between bases $2$ and $3$, and a multiplicative identity gives continuity and uniqueness.  The resulting base is exactly the one whose exponential function differentiates to itself. 

Beyond establishing the slope-based definition of $e$, we also
develop several consequences designed to show students why this particular base
deserves to be called natural.  Section~6 shows that once the derivative at the
origin is understood, the exponent law immediately determines the derivative of
$b^x$ at every point, and the distinguished choice $g(e)=1$ yields the
exceptionally simple identity
$\frac{d}{dx}e^x=e^x$.
From there, the logarithm and the general power rule fit naturally into the same
framework.  Section~7 then places $e$ in a broader mathematical context through
examples from optimization, data representation, combinatorics, and probability,
illustrating how the same constant emerges for reasons that at first appear
unrelated.  The appendices complement this development by making explicit several
points that are often passed over quickly in calculus: the construction of real
powers, the genuinely different methods required to differentiate integer,
rational, and irrational powers, and the elementary discrete optimization behind
the preference for summands equal to $3$.  Thus, the paper aims not only to justify
one definition of $e$, but also to use that definition as a gateway to a network
of ideas connecting analysis with other areas of mathematics.


\section{Three elementary tools}\label{sec:tools}

Our proof uses three ideas: a monotone-limit principle, convexity, and the IVT.  Each is standard in real analysis, but each can also be motivated geometrically in a calculus course; see, for example, Abbott \cite{Ab}.

\subsection{A monotone-limit principle}

The following is the one-sided functional version of the familiar fact that a bounded monotone sequence converges. For sake of clarity, a function $F$ is \textbf{non-decreasing} if $F(x) \leq F(y)$ whenever $x > y$.

\begin{theorem}[Monotone-limit principle]\label{thm:monotone-limit}
Suppose $F:(0,a]\to\mathbb{R}$ is non-decreasing and bounded below.  Then the right-hand limit
$$
 \lim_{x\to0^+}F(x)
$$
exists and equals $\inf\{F(x):0<x\le a\}$.
\end{theorem}

The picture behind the theorem is simple.  Because $F$ is non-decreasing, moving toward $0$ from the right can only move the function values downward; since they are bounded from below, they cannot decrease forever without approaching a finite level.  The two common obstructions to a finite limit, namely persistent oscillation and escape to infinity, have both been ruled out.

\begin{proof}
Let
$$
L\ :=\ \inf\{F(x):0<x\le a\}.
$$
Fix $\varepsilon>0$.  Since $L$ is the greatest lower bound, there is some $x_0\in(0,a]$ for which
$$
L\ \le\ F(x_0) \ <\ L+\varepsilon.
$$
Whenever $0<x<x_0$, monotonicity gives $F(x)\le F(x_0)$, while $L\le F(x)$ because $L$ is a lower bound.  Thus,
$$
L\ \le\ F(x)<L+\varepsilon,
$$
so $|F(x)-L|<\varepsilon$.  This is exactly the definition of $F(x)\to L$ as $x\to0^+$, where $x\to 0^+$ means that $x\to 0$ through positive values of $x$.
\end{proof}

\subsection{Convexity}

\begin{definition}\label{def:convex}
A function $f:I\to\mathbb{R}$ on an interval $I$ is \textit{convex} if, for all $x,y\in I$ and all $t\in[0,1]$,
\begin{equation}\label{eq:convex}
 f(tx+(1-t)y)\ \le\  t f(x)+(1-t)f(y).
\end{equation}
\end{definition}

Geometrically, the chord joining two points of the graph lies on or above the graph between those points (see Figure \ref{fig:convex}).

\begin{figure}[ht]
\centering
\begin{tikzpicture}[x=2.0cm,y=1.2cm]
  \draw[->] (-1.45,0) -- (1.65,0) node[right] {$x$};
  \draw[->] (0,-0.15) -- (0,2.55) node[above] {$y$};
  \draw[thick,domain=-1.25:1.45,samples=80,smooth] plot (\x,{\x*\x});
  \coordinate (A) at (-1,1);
  \coordinate (B) at (1.25,1.5625);
  \draw[thick,dashed] (A)--(B);
  \fill (A) circle (1.3pt) node[above left] {$(x_1,f(x_1))$};
  \fill (B) circle (1.3pt) node[above right] {$(x_2,f(x_2))$};
  \node at (0.92,0.35) {$y=x^2$};
\end{tikzpicture}
\caption{For a convex function, the graph lies below the chord joining any two of its points.}
\label{fig:convex}
\end{figure}
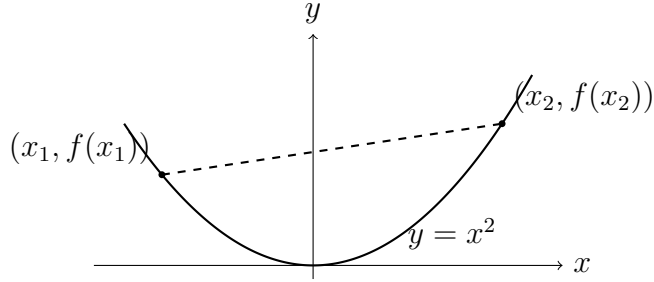

A convenient special case is midpoint convexity. In general, evaluation of functions at midpoints often gives useful information about the functions (see \cite{burden2011numerical}).

\begin{definition}\label{def:midpoint}
A function $f:I\to\mathbb{R}$ is \textit{midpoint convex} if
$$
 f\!\left(\frac{x+y}{2}\right)\ \le\ \frac{f(x)+f(y)}{2}\mbox{ for all }x,y\in I.$$
\end{definition}

Midpoint convexity alone is weaker than convexity because it requires the convexity inequality only at the midpoint, rather than at every point on the line segment.  Continuity supplies the missing step. 

\begin{lemma}\label{lem:midpoint-to-convex}
Every continuous midpoint-convex function on an interval is convex.
\end{lemma}

\begin{proof}
Midpoint convexity first gives \eqref{eq:convex} for $t=1/2$.  Applying the midpoint inequality again gives the cases $t=1/4$ and $t=3/4$.  For example,
$$
f\left(\frac{1}{4} x+\frac{3}{4} y\right)
\ \le\  \frac{1}{2} f\left(\frac{x+y}{2}\right)+\frac{1}{2} f(y)
\ \le\  \frac{1}{4} f(x)+\frac{3}{4} f(y).
$$
The case $t=3/4$ is analogous. Repeating the process gives coefficients with denominators $2,4,8,\ldots$. Thus, \eqref{eq:convex} holds for every $t$ of the form
$$
t\ =\ \frac{m}{2^n},\quad 0\le m\le 2^n.
$$
Numbers of this form are called \textit{dyadic rationals}.

Every $t\in[0,1]$ is the limit of a sequence of dyadic rationals $t_n$.  For each $n$,
$$
f(t_nx+(1-t_n)y)\ \le\ t_n f(x)+(1-t_n)f(y).
$$
Letting $n\to\infty$ and using continuity of $f$ on the left-hand side gives 
$$
f(tx+(1-t)y)\ \le\  t f(x)+(1-t)f(y),
$$
which is \eqref{eq:convex} for arbitrary $t$.
\end{proof}

\begin{remark}
Continuity is genuinely needed.  There exist discontinuous additive functions $A:\mathbb{R}\to\mathbb{R}$ satisfying 
$$
A\!\left(\frac{x+y}{2}\right)\ =\ \frac{A(x)+A(y)}{2},
$$
so they are midpoint convex \cite[p.~504]{GreenGustin}. However, such functions are not convex on $\mathbb{R}$ because every convex function is continuous on the interior of its domain
\cite[p.~68]{BoydVandenberghe}.  This example is optional background; no such pathology appears in the exponential functions considered below.
\end{remark}

We now record the convexity fact that drives the entire argument.

\begin{lemma}[Secant slopes of a convex function]\label{lem:secant}
Let $f$ be convex on an interval containing $0$, and define
$$Q(x)\ :=\ \frac{f(x)-f(0)}{x},\quad x\ne0.$$
Then $Q$ is non-decreasing on each side of $0$.  In particular, if $0<x<y$, then $Q(x)\le Q(y)$, and if $x<y<0$, then $Q(x)\le Q(y)$.
\end{lemma}

\begin{proof}
The statement says geometrically that, as the second endpoint of a chord moves to the right, the slope of the chord from $(0,f(0))$ cannot decrease.

Suppose first that $0<x<y$.  Since
$$x\ =\ \frac{x}{y}y+\left(1-\frac{x}{y}\right)0,$$
convexity gives
$$f(x)\ \le\  \frac{x}{y}f(y)+\left(1-\frac{x}{y}\right)f(0).$$
Subtracting $f(0)$ from both sides gives
$$
f(x)-f(0)\ \le\ \frac{x}{y}\bigl(f(y)-f(0)\bigr).
$$
Dividing by $x>0$ yields
$$
\frac{f(x)-f(0)}{x}\ \le\ \frac{f(y)-f(0)}{y},
$$
which is $Q(x)\le Q(y)$.

The case for $x<y<0$ is similar. 
\end{proof}

\subsection{The Intermediate Value Theorem}

We take the Intermediate Value Theorem as part of the standard first year calculus background  \cite{Ste}. A first course in analysis often covers a full proof of the theorem \cite[Exercise 4.5.1]{Ab}.

\begin{theorem}[Intermediate Value Theorem]\label{thm:ivt}
If $F$ is continuous on $[a,b]$ and a number $C$ lies between $F(a)$ and $F(b)$, then some $c\in[a,b]$ satisfies $F(c)=C$.
\end{theorem}

We will apply the IVT to $g$ only after establishing that it is well-defined and continuous.

\section{Why the slope $g(b)$ exists}\label{sec:existence}

We begin by proving that the exponential function is convex without differentiating it.

\begin{proposition}\label{prop:exp-convex}
For every $b>0$, the function $x\mapsto b^x$ is convex on $\mathbb{R}$.
\end{proposition}

\begin{proof}
The function is continuous by our standing convention.  Recall the two-variable arithmetic-geometric mean inequality (AM-GM): for $u,v\ge 0$,
\begin{equation}\label{e:AMGM2}
\frac{u+v}{2}\ \ge\ \sqrt{uv}.
\end{equation}
Notice \eqref{e:AMGM2} follows immediately from $(\sqrt{u}-\sqrt{v})^2\ge0$; see also, for example, \cite[p.~978, Exercise~49]{Ste}.  Applying AM-GM with $u=b^x$ and $v=b^y$ gives
$$
\frac{b^x+b^y}{2}\ \ge\ \sqrt{b^x b^y}\ =\  b^{(x+y)/2}.
$$
The final equality uses the real-exponent laws summarized in Appendix~\ref{app:real-exponents}.  Thus $b^x$ is midpoint convex, and Lemma \ref{lem:midpoint-to-convex} upgrades midpoint convexity to convexity.
\end{proof}

For $b>0$ and $h\ne0$, set
\begin{equation}\label{eq:qbh}
 q_b(h)\ :=\ \frac{b^h-1}{h}.
\end{equation}
By Lemma \ref{lem:secant}, $q_b(h)$ is non-decreasing in $h$ on $(0,\infty)$ and on $(-\infty,0)$.

\begin{theorem}\label{thm:g-exists}
For every $b>0$, the limit
$$g(b)\ =\ \lim_{h\to0}\frac{b^h-1}{h}$$
exists.
\end{theorem}

\begin{proof}
First suppose $b>1$.  For $h>0$, we have $q_b(h)>0$.  Since $q_b$ is non-decreasing as a function of $h$, Theorem \ref{thm:monotone-limit} shows that
$$
 L\ :=\ \lim_{h\to0^+}q_b(h)
$$
exists.

It remains to compare the two one-sided limits.  Set $h=-t$, where $t>0$.  Using the exponent laws,
$$q_b(-t)\ =\ \frac{b^{-t}-1}{-t}
 \ =\ \frac{1-b^{-t}}{t}\ =\ b^{-t}\frac{b^t-1}{t}
 \ =\ b^{-t}q_b(t).$$
As $t\to0^+$, continuity of $b^x$ gives $b^{-t}\to1$, and $q_b(t)\to L$.  Hence,
$$\lim_{h\to0^-}q_b(h)\ =\ L,$$
so the two-sided limit exists.

If $b=1$, then $q_1(h)=0$ for every $h\ne0$.  Finally, if $0<b<1$, let $c=1/b>1$.  Then
$$q_b(h)\ =\ \frac{c^{-h}-1}{h}\ =\ -q_c(-h),$$
and the existence of $\lim_{h\to0}q_c(h)$ implies the existence of $\lim_{h\to0}q_b(h)$.  Thus $g(b)$ exists for every $b>0$.
\end{proof}

Question \ref{q:existence} is now settled.  Notice that we did not compute $g(b)$ because doing so would require additional machinery concerning exponential and logarithmic functions, while our argument needs only qualitative properties of $g$, namely continuity and monotonicity.

\section{Putting the target slope between bases $2$ and $3$}\label{sec:bounds}

We answer Question \ref{q:bounds} without using any derivative formula.
The monotonicity of the secant slopes gives rigorous estimates using only convenient choices of $h$.  The fractions below are not magic.  To force $g(2)<1$, it is enough to find a positive $h$ for which $q_2(h)<1$; choosing $h=1/m$ asks for $2^{1/m}<1+1/m$, and $m=3$ is a small integer for which the comparison is immediate after cubing.  Similarly, to force $g(3)>1$ from the negative side, choosing $h=-1/m$ asks for $3^{1/m}>m/(m-1)$; $m=6$ makes this an elementary integer comparison.  This leads to the choices $h=1/3$ and $h=-1/6$ used below.

\begin{proposition}\label{prop:2-3}
We have
$$g(2)\ <\ 1\ <\ g(3).$$
\end{proposition}

\begin{proof}
Because $q_2$ is non-decreasing on $(0,\infty)$,
$$g(2)\ =\ \lim_{h\to0^+}q_2(h)\ \le\  q_2\left(\frac{1}{3}\right)\ =\ 3\left(2^{1/3}-1\right).$$
Now $2^{1/3}<4/3$ because $2<64/27$.  Therefore,
$$g(2)\ <\ 3\left(\frac{4}{3}-1\right)\ =\ 1.$$

For base $3$, the function $q_3$ is non-decreasing on $(-\infty,0)$.  Hence,
$$g(3)\ =\ \lim_{h\to0^-}q_3(h)\ \ge\  q_3\left(-\frac{1}{6}\right)
 \ =\ 6\left(1-3^{-1/6}\right).$$
Moreover,
$$3\ >\ \left(\frac65\right)^6$$
because $3\cdot5^6=46875>46656=6^6$.  Thus, $3^{1/6}>6/5$, so $3^{-1/6}<5/6$, and consequently,
$$g(3)\ >\ 6\left(1-\frac56\right)\ =\ 1.$$
\end{proof}

\section{Continuity, uniqueness, and the definition of $e$}\label{sec:continuity}

A simple identity reveals the structure of $g$.

\begin{proposition}[A logarithmic law]\label{prop:loglaw}
For all $a,b>0$,
\begin{equation}\label{eq:loglaw}
 g(ab)\ =\ g(a)+g(b).
\end{equation}
In particular, $g(1)=0$ and $g(1/b)=-g(b)$.
\end{proposition}

\begin{proof}
Using $(ab)^h=a^h b^h$, which holds for real $h$ by Appendix~\ref{app:real-exponents}, we first expand
$$
(ab)^h-1\ =\ a^hb^h-1\ =\ a^h(b^h-1)+(a^h-1).
$$
After dividing by $h\ne0$,
$$
\frac{(ab)^h-1}{h}\ =\ a^h\frac{b^h-1}{h}+\frac{a^h-1}{h}.
$$
As $h\to0$, continuity of the real exponential gives $a^h\to1$, while the two quotients tend to $g(b)$ and $g(a)$ by definition.  Taking limits therefore gives \eqref{eq:loglaw}.  Since $q_1(h)=0$, we have $g(1)=0$.  Finally, $1=b(1/b)$, so \eqref{eq:loglaw} gives 
$$0\ =\ g(1)\ =\ g(b)+g(1/b),$$ 
and therefore, $g(1/b)=-g(b)$.
\end{proof}

The same identity makes continuity almost automatic.

\begin{proposition}\label{prop:g-continuous}
The function $g:(0,\infty)\to\mathbb{R}$ is continuous.
\end{proposition}

\begin{proof}
We first prove continuity at $1$.  Recall from \eqref{eq:qbh} that $q_b(h)=(b^h-1)/h$.  The secant estimates at $h=-1$ and $h=1$ give, for every $b>0$,
$$
q_b(-1)\ \le\ g(b)\ \le\ q_b(1).
$$
Since
$$
q_b(-1)\ =\ 1-\frac{1}{b}
\quad\mbox{ and }\quad
q_b(1)\ =\ b-1,
$$
we obtain the useful two-sided estimate
\begin{equation}\label{eq:g-squeeze}
1-\frac{1}{b}\ \le\ g(b)\ \le\ b-1.
\end{equation}
Both bounds tend to $0$ as $b\to1$, so the Squeeze Theorem gives $g(b)\to0=g(1)$.

Now fix $b^*>0$.  For $b$ near $b^*$, write $b=b^*(b/b^*)$.  Proposition \ref{prop:loglaw} gives
$$
g(b)\ =\ g\left(b^*\frac{b}{b^*}\right)\ =\ g(b^*)+g\left(\frac{b}{b^*}\right),
$$
so, after subtracting $g(b^*)$,
$$
 g(b)-g(b^*)\ =\ g\left(\frac{b}{b^*}\right).
$$
As $b\to b^*$, we have $b/b^*\to1$, so continuity at $1$ implies $g(b)\to g(b^*)$.
\end{proof}

We also need to show that $e$ is the unique base with $g(e) = 1$.  For $c>1$, Lemma \ref{lem:secant} gives
$$g(c)\ =\ \lim_{h\to0^-}q_c(h)\ \ge\ q_c(-1)\ =\ 1-\frac{1}{c}\ >\ 0.$$
Therefore, if $0<a<b$, then $b/a>1$ and
$$g(b)-g(a)\ =\ g(b/a)\ >\ 0.$$
Therefore, $g$ is strictly increasing.

We can now complete Stewart's geometric construction.

\begin{theorem}\label{thm:e}
There is a unique number $e\in(2,3)$ such that
$$g(e)\ =\ 1.$$
Equivalently, $e$ is the unique positive base for which the graph of $y=e^x$ has tangent slope $1$ at $(0,1)$.
\end{theorem}

\begin{proof}
By Proposition \ref{prop:2-3}, $g(2)<1<g(3)$.  Proposition \ref{prop:g-continuous} and the IVT give some $e\in(2,3)$ with $g(e)=1$.  Since $g$ is strictly increasing, that $e$ is unique.
\end{proof}

Thus all three questions from the introduction have been answered.

\section{Why this base is natural for calculus}\label{sec:natural}

The definition $g(e)=1$ is not merely aesthetically pleasing.  It immediately produces the cleanest possible derivative rule for an exponential function.

There is also a useful structural lesson here: for an exponential function, knowing the derivative at a single point, namely $x=0$, gives the derivative everywhere.  The reason is the exponent law $b^{x+h}=b^x b^h$, which factors every difference quotient into a fixed factor $b^x$ and the same quotient that appears at the origin.  Students may recognize a similar phenomenon from trigonometric differentiation: angle-addition formulas reduce the derivatives of $\sin x$ and $\cos x$ to special limits at $0$, such as $\lim_{h\to0}\sin h/h=1$.

\begin{proposition}\label{prop:derivative}
For every $b>0$ and every real $x$,
\begin{equation}\label{eq:derivative-bx}
 \frac{d}{dx}b^x\ =\ g(b)b^x.
\end{equation}
In particular,
\begin{equation}\label{eq:derivative-ex}
 \frac{d}{dx}e^x\ =\ e^x.
\end{equation}
\end{proposition}

\begin{proof}
Using the exponent law,
$$\frac{b^{x+h}-b^x}{h}\ =\ b^x\frac{b^h-1}{h}.$$
Letting $h\to0$ gives \eqref{eq:derivative-bx}.  For $b=e$, Theorem \ref{thm:e} gives $g(e)=1$, yielding \eqref{eq:derivative-ex}.
\end{proof}

This is the fundamental reason $e$ deserves the adjective \textit{natural}: differentiation leaves the exponential function $e^x$ unchanged, unlike any other exponential function.  Once the natural logarithm $\ln x$ is introduced as the inverse of $e^x$, its derivative follows from the chain rule.  Since
\begin{equation}\label{e: eid}
e^{\ln x}\ =\ x, \quad x>0,
\end{equation}
differentiating both sides gives
$$
e^{\ln x}(\ln x)'\ =\ 1.
$$
Applying \eqref{e: eid}, we obtain
$$
\frac{d}{dx}\ln x\ =\ \frac1x.
$$

We can now differentiate $x^{\sqrt{3}}$ using the chain rule. Note that we cannot mimic the standard proofs for differentiating $x^2$ or $x^3$, since the usual binomial theorem applies only to nonnegative integer exponents.
For any real number $\alpha$ and $x>0$,
$x^\alpha=e^{\alpha\ln x}$.
Therefore, by the chain rule,
$$
\frac{d}{dx}x^\alpha
\ =\
e^{\alpha\ln x}\frac{\alpha}{x}
\ =\ 
\alpha x^{\alpha-1}.
$$
Thus, the earlier question of how to differentiate $x^{\sqrt{3}}$ is no longer mysterious. A detailed comparison of the methods for differentiating different powers is given in Appendix \ref{app:three-powers}.

This viewpoint also connects to the other standard definitions of $e$.  Once the inverse-function theorem and the Fundamental Theorem of Calculus are available, $\ln e=1$ becomes the area characterization \eqref{eq:area}.  Taylor's theorem applied to $e^x$ yields \eqref{eq:series}.  Standard logarithmic limit arguments then recover the compound-interest formula \eqref{eq:compound}.  In this sense, the tangent-slope definition is not a competing version of $e$; it is an entry point into the same theory.

\section{Appearances of $e$ after the construction}\label{sec:appearances}

 We close with two examples in which $e$ appears for conceptually different reasons.

\subsection{Breaking a sum to maximize a product}

Suppose a positive total $S$ is split into $n$ positive real pieces
$$x_1+\cdots+x_n\ =\ S.$$
For fixed $n$, the $n$-variable AM-GM inequality \cite[p.~978, Exercise~49]{Ste} says
$$
\frac{x_1+\cdots+x_n}{n}\ \ge\ (x_1\cdots x_n)^{1/n}.
$$
Because $x_1+\cdots+x_n=S$, raising both sides to the $n$\textsuperscript{th} power gives
$$
x_1\cdots x_n\ \le\ \left(\frac{S}{n}\right)^n,
$$
with equality exactly when all pieces are equal.  Thus, for fixed $n$, equal pieces maximize the product.
Now temporarily allow $n$ to vary continuously and consider
$$
 P(n)\ =\ \left(\frac{S}{n}\right)^n.
$$
Taking logarithms is convenient because $\ln$ is strictly increasing, so maximizing $P(n)$ is equivalent to maximizing $\ln P(n)$.  We have
$$
 \ln P(n)\ =\ n\ln\left(\frac{S}{n}\right)
 \ =\ n(\ln S-\ln n).
$$
Treat $n>0$ temporarily as a real variable.  Differentiating with the product rule gives
\begin{align*}
\frac{d}{dn}\ln P(n)
 &\ =\ (\ln S-\ln n)+n\left(-\frac1n\right)\\
 &\ =\ \ln S-\ln n-1.
\end{align*}
Thus, a critical point must satisfy
$$
 \ln S-\ln n-1\ =\ 0.
$$
Equivalently, $\ln(S/n)=1$, so $S/n=e$ and hence $n=S/e$.  Moreover,
$$
\frac{d^2}{dn^2}\ln P(n)\ =\ -\frac1n<0,
$$
so this critical point is indeed a maximum.  Thus the optimal common piece in the continuous setting is
$$
 \frac{S}{n}\ =\ e.
$$
There is also a classical integer version: when one must split an integer into positive integer summands, the maximizing decomposition uses as many $3$'s as possible, with a small adjustment when the remainder is $1$.  Appendix~\ref{app:integer-product} gives the short elementary argument.  The continuous calculation above explains conceptually why $3$ is the preferred integer: it is the integer closest to $e$. 

This optimization principle also has an interpretation in data storage. Suppose a system consists of several independent components, where the $i$th component can be in $x_i$ possible states. Then the total number of distinct configurations that the system can represent is
$x_1x_2\cdots x_n$.
Thus, if the total number of available states $x_1+\cdots+x_n$ is fixed, maximizing the product amounts to maximizing the number of configurations that can be stored. 

Last but not least, a closely related phenomenon appears in the problem of choosing an
\emph{optimal radix} for representing numbers. Under one natural model,
the asymptotic cost of representing numbers in base $b$ is proportional to $\frac{b}{\ln b}$.
If $b$ is allowed to vary continuously, this quantity is minimized at
$b=e$; among integer bases, the optimum is therefore attained at $b=3$.
This provides another appearance of the same principle: a continuous
optimization singles out $e$, while the corresponding discrete problem
naturally favors $3$ \cite{OptimalRadix}.

\subsection{Derangements and the probability $1/e$}

Imagine $n$ labeled chocolates placed randomly into $n$ labeled positions, one chocolate per position.  What is the probability that every chocolate lands in the wrong position?  A \textit{derangement} is a permutation with no fixed points; let $D_n$ denote the number of derangements of $n$ objects.  Since there are $n!$ permutations in total, the desired probability is $D_n/n!$.  

To compute $D_n$, let $A_i$ be the set of permutations in which the $i$\textsuperscript{th} chocolate is in its correct position. We want to count permutations belonging to none of the sets
$A_1,\dots,A_n$. By the Inclusion-Exclusion Principle,
$$D_n\ =\ n!
-\binom n1 (n-1)!
+\binom n2 (n-2)!
-\cdots
+(-1)^n.$$
Indeed, there are $(n-1)!$ permutations fixing a specified chocolate,
$(n-2)!$ permutations fixing two specified chocolates, and, more generally,
$(n-k)!$ permutations fixing any specified set of $k$ chocolates. Hence
$$D_n \ =\ \sum_{k=0}^n (-1)^k \binom nk (n-k)!.$$
Since
$$\binom nk (n-k)!\ =\ \frac{n!}{k!},$$
we obtain
$$D_n\ =\ n!\sum_{k=0}^n \frac{(-1)^k}{k!}.$$
Dividing by $n!$ gives
$$
\frac{D_n}{n!}\ =\ 
\sum_{k=0}^n \frac{(-1)^k}{k!}.$$
Therefore, as $n\rightarrow\infty$,
$$\frac{D_n}{n!}\ \longrightarrow\
\sum_{k=0}^{\infty}\frac{(-1)^k}{k!}\ =\ e^{-1}.$$

\subsection{Further appearances of $e$}
A classical appearance is Stirling's approximation
$$
 n!\ \sim\ \sqrt{2\pi n}\left(\frac{n}{e}\right)^n,
$$
which gives a remarkably accurate description of the growth of factorials; see Robbins \cite{Rob} for a concise refinement.  

Furthermore, the second named author encountered $e$ repeatedly in standard
probabilistic settings, including the Poisson and exponential distributions
and moment generating functions \cite{MillerProbability}. Related chess-inspired problems provide
further striking examples. Miller et al.\ \cite{MillerShengTurek} showed that when $n$ rooks are placed randomly
on an $n\times n$ chessboard, the expected proportion of safe squares tends to
$1/e^2$. More recently, Cashman et al.\ \cite{CashmanCooperMarquezMillerShuffelton} extended
this phenomenon to higher-dimensional chess pieces; in particular, for
$k$-dimensional hyper-rooks, the limiting proportion of safe squares is
$1/e^k$.

These examples illustrate a broader theme: $e$ is not natural because mathematicians chose a convenient symbol for one isolated limit.  It is natural because the same constant repeatedly emerges from growth, differentiation, optimization, counting, and probability. This illustrates one advantage of studying a broad range of areas in mathematics: a wide mathematical perspective is often needed to recognize connections between seemingly different topics.

\section*{Declaration of generative AI use} 
The authors conceived the present paper, its structure, and mathematical content independently. During the preparation of the paper, ChatGPT (GPT-5.6 Sol) \cite{GPT} was subsequently used to assist with revising and polishing the exposition, reorganizing portions of the paper, suggesting and summarizing references, and checking presentations of some mathematical arguments. The authors take full responsibility for the final content and for verifying all mathematical statements, proofs, and references.

\appendix

\section{What does $b^x$ mean for a real exponent?}\label{app:real-exponents}

For integer exponents, $b^n$ is defined by repeated multiplication, and for rational exponents $p/q$, one defines
$$
b^{p/q}\ =\ \left(\sqrt[q]{b}\right)^p,
$$
with $b>0$.  The remaining step is to define $b^x$ when $x$ is irrational.

One concrete route is to approximate $x$ by rational numbers.  Choose any rational sequence $r_n\to x$: for example, the truncations of the decimal expansion of $x$, and set
\begin{equation}\label{eq:real-exp-appendix}
b^x\ :=\ \lim_{n\to\infty}b^{r_n}.
\end{equation}
To make this a definition, two facts must be checked: that the limit exists, and that it does not depend on the chosen rational sequence.  A useful elementary estimate, proved for example in Badi \cite[Lemma 2]{Badi}, says that for each fixed $b>0$, there is a constant $C_b$ such that
$$
|b^q-1|\ \le\ C_b|q|
$$
for rational $q$ sufficiently close to $0$.  Consequently, if $r_m$ and $r_n$ are close, then
$$
\left|\frac{b^{r_m}}{b^{r_n}}-1\right|
 \ =\ |b^{r_m-r_n}-1|
$$
is small.  This is the key estimate showing that the rational approximations give a well-defined limit.

Once well-definedness is established, the exponent law follows naturally.  If $r_n\to x$ and $s_n\to y$ are rational sequences, then $r_n+s_n\to x+y$, so
\begin{align*}
b^{x+y}
 &\ =\ \lim_{n\to\infty} b^{r_n+s_n}\\
 &\ =\ \lim_{n\to\infty} b^{r_n}b^{s_n}\\
 &\ =\ b^x b^y.
\end{align*}
The same construction gives continuity of $x\mapsto b^x$.  These are the properties used in the main text.  A full development can be found in \cite{Badi}; Eidswick \cite{Eidswick} begins from the same continuous extension of rational powers.

\section{Three powers, three methods}\label{app:three-powers}

The three functions
$$
x^3,\quad x^{3/2},\quad x^{\sqrt3}
$$
look similar, but at an early stage of calculus, studying their derivatives naturally leads to different methods.

\subsection*{An integer exponent}
For $f(x)=x^3$, the difference quotient and the binomial theorem give
\begin{align*}
\frac{(x+h)^3-x^3}{h}
 &\ =\ \frac{3x^2h+3xh^2+h^3}{h}\\
 &\ =\ 3x^2+3xh+h^2.
\end{align*}
Letting $h\to0$ yields $f'(x)=3x^2$.  No exponential or logarithmic functions are needed.

\subsection*{A rational exponent}
For $x>0$, let $f(x)=x^{3/2}$.  Once differentiability of the relevant root function has been established, the algebraic relation
$$
f(x)^2\ =\ x^3
$$
allows us to use implicit differentiation.  Equivalently, apply the product rule to $f(x)f(x)=x^3$:
$$
f'(x)f(x)+f(x)f'(x)\ =\ 3x^2,
$$
so
$$
2f(x)f'(x)\ =\ 3x^2.
$$
Since $f(x) = x^{3/2}$,
$$
f'(x)\ =\ \frac{3x^2}{2x^{3/2}}\ =\ \frac32x^{1/2}.
$$
The same idea works for a general rational exponent $p/q$: if $y=x^{p/q}$, then $y^q=x^p$, so differentiation reduces the problem to integer powers together with the product rule.

Another way to differentiate $x^{3/2}$ is to work directly from the definition of the derivative. The key idea is to multiply by a carefully chosen factor so that the fractional powers disappear. In this case, multiplying by the conjugate-like expression $(x+h)^{3/2}+x^{3/2}$ produces a difference of squares, turning the numerator into $(x+h)^3-x^3$, which can then be handled using the usual polynomial derivative. In particular,
\begin{align*}
f'(x)&\ =\ 
\lim_{h\to 0}
\frac{(x+h)^{3/2}-x^{3/2}}{h}
\cdot
\frac{(x+h)^{3/2}+x^{3/2}}
{(x+h)^{3/2}+x^{3/2}}\\
&\ =\ 
\lim_{h\to 0}
\frac{(x+h)^3-x^3}{h}
\cdot
\lim_{h\to 0}
\frac{1}{(x+h)^{3/2}+x^{3/2}}\\
&\ =\
3x^2\cdot\frac{1}{2x^{3/2}}
\ =\
\frac32x^{1/2}.
\end{align*}
This trick works because $3/2$ is rational: by raising suitable fractional powers to an integer power, we can reduce the calculation to familiar algebraic identities. It naturally raises the question of whether a similar trick could work for an irrational exponent such as $\pi$; there, no finite algebraic manipulation of this kind will eliminate the irrational power.

\subsection*{An irrational exponent}
For $f(x)=x^{\sqrt3}$, there is no polynomial relation $y^q=x^p$ with integers $p$ and $q$ that plays the same role.  We first need a meaning for arbitrary real exponents.  Once the exponential and logarithmic functions are available, we may write
$$
x^{\sqrt3}\ =\ e^{\sqrt3\ln x},\quad x>0,
$$
and the chain rule gives
$$
\frac{d}{dx}x^{\sqrt3}
 \ =\ e^{\sqrt3\ln x}\frac{\sqrt3}{x}
 \ =\ \sqrt3\,x^{\sqrt3-1}.
$$
This is why the theory of exponential and logarithmic functions completes the power rule for arbitrary real exponents.

\section{Why $3$'s maximize the integer product}\label{app:integer-product}

For completeness, here is the elementary argument behind the integer version mentioned in Section~\ref{sec:appearances}.  Suppose a positive integer $S$ is written as a sum of positive integers and we want their product to be as large as possible.

First, no summand $m\ge5$ can occur in an optimal decomposition: replacing $m$ by $2$ and $m-2$ keeps the sum unchanged but changes the product factor from $m$ to
$$
2(m-2)\ >\ m\quad (m>4).
$$
Thus, every summand may be taken from $\{2,3,4\}$.  A $4$ may be replaced by $2+2$ without changing the product, since $4=2\cdot2$.  We are therefore reduced to using only $2$'s and $3$'s.

Finally, three $2$'s should never appear together, because
$$
2+2+2\ =\ 3+3
\quad\text{but}\quad
2^3\ =\ 8\ <\ 9\ =\ 3^2.
$$
Hence, an optimal decomposition uses as many $3$'s as possible, except that a remainder of $1$ is avoided by replacing $3+1$ with $2+2$.  This gives the following rule: use $3$'s, with one $2$ if the remainder is $2$ and two $2$'s if the remainder is $1$.

\end{document}